\documentclass[12pt,a4paper,reqno]{amsart}
\usepackage{geometry}
\usepackage{amsmath,amssymb}
\usepackage{amsfonts}
\usepackage{eucal}
\usepackage{amsthm}
\usepackage{mathrsfs}
\numberwithin{equation}{section}
\usepackage[pagewise]{lineno}

\def\a{\alpha}
\def\b{\beta}

\def\d{\delta}
\def\e{\varepsilon}
\def\f{\varphi}
\def\g{\psi}

\def\x{\xi}

\def\re{\mathbb{R}}

\def\pa{\partial}

\renewcommand{\Re}{\text{{\rm Re}\;}}
\renewcommand{\Im}{\text{{\rm Im}\;}}

\newcommand{\supp}{\text{{\rm supp}\;}}

\newtheorem{thm}{Theorem}[section]
\newtheorem{lem}[thm]{Lemma}
\newtheorem{prop}[thm]{Proposition}
\newtheorem{cor}[thm]{Corollary}

\theoremstyle{definition}

\newtheorem{ass}{Assumption}

\newtheorem{ex}{Example}

\theoremstyle{remark}
\newtheorem{rem}[thm]{Remark}

\title[]%
{Smoothness of the fundamental solution of Schr\"odinger equations with mild trapping}

\author{Kouichi Taira}
\address{Department of Mathematical Sciences, Ritsumeikan University, 1-1-1 NojiHigashi, Kusatsu, 525-8577 Japan}
\email{ktaira@fc.ritsumei.ac.jp}
\date{}

\begin{document}
\maketitle

\begin{abstract}
In this short note, smoothness of the fundamental solution of Schr\"odinger equations on a complete manifold is studied. It is shown that
\begin{itemize}
\item the fundamental solution is smooth under ``mild" trapping conditions;
\item there is a Riemannian manifold which is equal to Euclidean space outside a compact set such that the fundamental solution is not smooth.
\end{itemize}

\end{abstract}


\section{Introduction}

\subsection{Background}

Let $(M,g)$ be a complete (non-compact) Riemannian manifold. As is well-known, the Laplace operator $-\Delta_g$ is essentially self-adjoint on $C_c^{\infty}(M)$. We denote the integral kernel of $e^{it\Delta_g}$ by $E_t$:
\begin{align*} 
E_t(x,y):=e^{it\Delta_g}(x,y)\quad t\in \re,\quad (x,y)\in M\times M.
\end{align*}
We say that $E_t$ is the fundamental solution of the Schr\"odinger equation.

It is known that due to these dispersive properties, Schr\"odinger equations have smoothing properties such as the local smoothing effect, which states that the map $L^2(M)\ni u_0\to e^{it\Delta_g}u_0\in L^2_{loc}(\re_t; H_{loc}^{\frac{1}{2}}(M))$ is continuous when $(M,g)$ is Euclidean space.
In general, Doi showed that the local smoothing effect holds if and only if the trapped set of the geodesic flow is empty \cite{D1}. For recent progress of the local smoothing effect and semiclassical resolvent estimates, see the survey \cite{W}.

In this short note, we study another smoothing property, that is, the smoothness of the fundamental solution. Smoothness of an evolution equation does not hold in general. Indeed, the wave propagators $\cos t\sqrt{\Delta_g}, \frac{\sin t\sqrt{\Delta_g}}{\sqrt{-\Delta_g}}$ are not smooth at all due to the lack of the dispersive properties.
As is shown in \cite[Theorem 1.5]{D2}, for each $t\neq 0$, the fundamental solution $E_t$ for the Schr\"odinger equation is smooth under the non-trapping condition. More precisely, the wavefront set of $E_t$ is contained in a subset described by the forward/backward trapped sets and the zero section of $T^*M$. However, as far as the author is aware of, it is in general not known whether $E_t$ is smooth in the presence of trapped trajectories of the geodesic flow. A purpose of this note is to show that
\begin{itemize}
\item if the trapped set of the geodesic flow is ``mild" enough (for example, if the trapped set is hyperbolic with the negative topological pressure), then $E_t$ is smooth;
\item there is a Riemmanian manifold $(M,g)$, which is equal to Euclidean space outside a compact set such that $E_t$ is not smooth.
\end{itemize}

When $M$ is compact, the smoothness of the fundamental solution is studied for the sphere (\cite{KR} or \cite{T}) or the interval. On the circle, it is shown in \cite{KR} that $E_t$ is not smooth and the precise regularity of $E_t$ is derived. In \cite{T}, an explicit formula for $E_t$ on the sphere is given at rational times.
 In \cite[Remark 4]{Y}, the author points out that the fundamental solution is nowhere integrable (in particular, not smooth) for the Dirichlet Laplacian on the interval $[0,\pi]$. In general, the fundamental solution on compact Riemannian manifolds is not smooth due to the existence of eigenfunctions with high-frequency (see Remark \ref{compnotsm}). From this observation, even when $M$ is non-compact, it is expected that a strongly scared quasimode (an approximate eigenfunction) disturbs the smoothness of $E_t$. This is the key idea of the proof of our main Theorem \ref{mainth2}.

For a Schr\"odinger operator $-\Delta +V(x)$ with a real-valued smooth potential $V$ on $\re^n$, the smoothness of the fundamental solution $E_t$ strongly depends on the growth rate of $V$. If $V$ is at most quadratic and if $t\neq 0$ is sufficiently small, then $E_t$ is smooth (\cite{F}). Moreover, if $V$ is sub-quadratic, then $E_t$ is smooth (\cite[Theorem 1.1]{Y}) for all time $t\neq 0$. On the other hand, it is shown in \cite[Theorem 1.2]{Y} that $E_t(x,y)$ is nowhere $C^1$ with respect to $(t,x,y)$ if $V$ is super-quadratic and if the dimension is one. In the super-quadratic case with higher dimensions, the problem has not been solved so far.

\subsection{Main results}

\begin{ass}\label{assresfree}
Suppose that there exist $\e_0,h_0,C_0>0$ and $\b<2$ such that for all $0<h<h_0$ and $\chi\in C_c^{\infty}(M)$, the outgoing resolvent $\chi(-h^2\Delta_g-z)^{-1}\chi\in B(L^2(M))$ for $\Im z>0$ has an analytic extension to a region
\begin{align*}
D_h:=\{z\in \mathbb{C}\mid \Re z\in [1-\e_0,1+\e_0],\,\, \Im z\geq -C_0h^{\b}\}.
\end{align*}
Moreover, there exists $\a>0$ such that the following holds: For each $\chi\in C_c^{\infty}(M)$, there exists $C_{\chi}>0$ such that
\begin{align*}
\|\chi (-h^2\Delta_g-z)^{-1} \chi\|\leq C_{\chi}h^{-\a}
\end{align*}
for all $z\in D_h$.

\end{ass}

In \cite{I}, the author studies the semiclassical behavior of distorted plane waves with hyperbolic trapping under existence of a resonance free strip similar to Assumption \ref{assresfree}. This result is an extension of his earlier results to this setting.

Now we define the trapped set of the geodesic flow. We denote the geodesic flow by $\f_t$. The trapped set is defined by
\begin{align*}
K:=\{(x,\x)\in T^*M\mid |\x|_g=1,\,\, \f_t(x,\x)\,\, \text{remains in a bounded set for all $t\in \re$} \}.
\end{align*}

\begin{ex}
Suppose a complete Riemmanian manifold $(M,g)$ is isometric outside a compact set
 to copies of Euclidean space or even asymptotically hyperbolic spaces (for these definition, see \cite[\S 4.1, 4.2]{DV} and \cite[\S 3.1]{NZ}. Actually, we can deal with more general asymptotically conic spaces, see \cite[\S 3.2]{NZ}.)
Moreover, suppose one of the following conditions holds (see \cite{NZ}, \cite{BD}, \cite{NZ2} and \cite{C} respectively):
\begin{itemize}
\item[$(1)$] $K$ is a hyperbolic trapped set and the topological pressure of the geodesic flow satisfies $\mathcal{P}(\frac{1}{2})<0$ (for the definition of the topological pressure, see \cite[before Theorem 3]{NZ} with $p(x,\x)=|\x|_g^2$ and $E=1$);

\item[$(2)$] there is a geodesically convex neighborhood $U$ of $K$ such that $U$ is isometric to a geodesically convex neighborhood of the trapped set in a convex co-compact hyperbolic surface;

\item[$(3)$] $K$ is a normally hyperbolic trapped set in the sense of \cite{NZ2};

\item[$(4)$] $(M,g)$ is a warped product manifold and $K$ is a disjoint union of unstable finitely degenerate trapped sets (see \cite{CW} or \cite[\S 3.2]{C}) and finitely degenerate inflection transmission trapped sets (see \cite{CM} or \cite[\S 3.3]{C}).
\end{itemize}
Then Assumption \ref{assresfree} is satisfied. This is checked by the results in \cite{NZ}, \cite{BD}, \cite{NZ2} or \cite{C} and by the gluing of semiclassical resolvent estimates \cite[Appendix A]{C} or \cite[Theorem 2.1]{DV}. See also \cite[\S 6]{DV}.
Moreover, we can take $\b=1$ in the cases $(1), (2)$ or $(3)$.
\end{ex}

Now we state our first main theorem.

\begin{thm}\label{mainth1}
Under Assumption \ref{assresfree}, we have $E_t\in C^{\infty}(M\times M)$ for each $t\neq 0$.
\end{thm}

\begin{rem}
We note that our result (Theorem \ref{mainth1}) does not contradict the lack of smoothing effects which is shown in \cite[Theorem 1.2]{D2}. In fact, the proof in \cite[Theorem 1.2]{D2} shows that the smoothing effects (\cite[$(i)_r, (ii)_r$]{D2}) break when the time $|t|$ is small enough depending on frequency $1/h$. On the other hand, the time $t\neq 0$ is fixed in our setting.
\end{rem}

It is well-known that the size of the resonance free strip is closely related to how ``mild" the trapped set is. In fact, the size of the resonance free strip $C_0h^{\b}$ corresponds to the inverse of the lifetime of a particle. Hence, the bigger the latter quantity is, the ''milder" the trapped set is. Thus, Theorem \ref{mainth1} says that when the trapped set is ''mild" enough, then the fundamental solution becomes smooth.

In \cite{D2}, the smoothness of $E_t$ under the non-trapping condition is a consequence of the following smoothing effects (\cite[$(ii)_r$]{D2}): If $u_0$ is a compactly supported distribution and $A\in \Psi_{cpt}^{r}$ (where $\Psi_{cpt}^r$ is a class of compactly supported pseudodifferential operators) with $r\geq 0$, then
\begin{align*}
L^2_{cpt}(M)\ni u_0\mapsto |t|^rAe^{it\Delta_g}u_0\in C(\re; H^{r}(M))
\end{align*}
is continuous.
However, as is shown in \cite[Theorem 1.2]{D2}, this is false when the trapped set is not empty. Here, we use the existence of a resonance free strip under a ``mild" trapping condition in order to prove the smoothness. The result for the non-trapping case is also reproved by our method (Theorem \ref{mainth1}) and the existence of a resonance free strip \cite{M}.

Next, we state our negative result.

\begin{thm}\label{mainth2}
There exist a Riemmanian manifold $(M,g)$ which is equal to Euclidean space $\re^n$ outside a compact set such that $E_t\notin C^{\infty}(M\times M)$ for all $t\in \re$.
\end{thm}

The proof of Theorem \ref{mainth2} depends on the well-known fact that an elliptic closed geodesic admits the strongly scared quasimode on an asymptotically conic manifold (\cite{C} or \cite{RT}).


\noindent
\textbf{Acknowledgment.}  
The author would like to thank Hans Christianson, Wataru Nakahashi and St\'ephane Nonnenmacher for helpful discussions.

\section{Smoothness of the fundamental solution}\label{sec2}

In this section, we prove Theorem \ref{mainth1}.
We may assume $t>0$.
Let $R_{\pm}(z)$ be a meromorphic continuation of the outgoing/incoming $(-h^2\Delta_g-z)^{-1}$ for $\pm \Im z>0$.

\begin{prop}\label{smprop}
Suppose that Assumption \ref{assresfree} is fulfilled. 
Let $\chi\in C_c^{\infty}(M)$ and $\g\in C_c^{\infty}((0,\infty))$. Then, for $h\in (0,h_0]$ and $t\geq 0$, we have
\begin{align*}
\|\chi e^{it\Delta_g}\g(-h^2\Delta_g)\chi\|_{L^2\to L^2}\leq&Ch^{-\a}e^{-C_0h^{\b-2}t}+O(h^{\infty}),
\end{align*}
where $O(h^{\infty})$ is uniformly in $t\geq 0$. In particular, for fixed $t>0$, we have
\begin{align*}
\|\chi e^{it\Delta_g}\g(-h^2\Delta_g)\chi\|_{L^2\to L^2}=O(h^{\infty})
\end{align*}
since $\b<2$.
\end{prop}

\begin{proof}
We follow the argument in \cite[Theorem 7.15, Theorem 7.16]{DZ} or \cite[Lemma 4.2]{I}.
Let $\tilde{\g}$ be an almost analytic continuation (\cite[Theorem 3.6]{Z}) of $\g$: $\pa_{\bar{z}}\g=O(|\Im z|^{\infty})$ and $\supp\tilde{\g}\subset \{z\in\mathbb{C}\mid \Re z\in \supp \g \}$. By the Green formula, we have
\begin{align*}
\chi e^{ish\Delta_g}\g(-h^2\Delta_g)\chi=&\frac{1}{2\pi i}\int_{\re}e^{-i\frac{s}{h}z}\chi (R_+(z)-R_-(z))\chi \g(z) dz\\
=&\frac{1}{2\pi i}\int_{\Im z=-C_0h^{\b}}e^{-i\frac{s}{h}z}\chi (R_+(z)-R_-(z))\chi \tilde{\g}(z) dz\\
&+\frac{1}{2\pi i}\int_{-C_0h^{\b}\leq \Im z\leq 0}e^{-i\frac{s}{h}z}\chi (R_+(z)-R_-(z))\chi \pa_{\bar{z}}\tilde{\g}(z) dz.
\end{align*}
Thus, since $|e^{-i\frac{s}{h}z}|\leq 1$ for $\Im z\leq 0$ and $s\geq 0$, we have
\begin{align*}
\|\chi e^{ish\Delta_g}\g(-h^2\Delta_g)\chi\|_{L^2\to L^2}\leq& Ch^{-\a}e^{-C_0h^{\b-1}s}+Ch^{-\a}O((h^{\b})^{\infty})\\
=&Ch^{-\a}e^{-C_0h^{\b-1}s}+O(h^{\infty}).
\end{align*}
Setting $s=t/h$, we obtain the desired result.

\end{proof}

From this proposition, we immediately obtain the following corollary.

\begin{cor}\label{intsmcor}
Under Assumption \ref{assresfree}, for each $t>0$, $N\in \re$ and $\chi\in C_c^{\infty}(M)$, the operator $\chi e^{it\Delta_g}\chi$ maps from $H^{-N}(M)$ to $H^N(M)$ continuously. 
\end{cor}

We also need the following elementary lemma.

\begin{lem}\label{intsmlem}
Suppose that a continuous linear operator $A:C_c^{\infty}(M)\to \mathcal{D}'(M)$ can be extended to an operator which maps from $H^{-N}(M)$ to $H^N(M)$ continuously for each $N>0$. Then the Schwarz kernel of $A$ is smooth.
\end{lem}

\begin{proof}
Let $k$ be a non-negative integer.
Since the delta function $a\in M\mapsto \d_a(x):=\d(x-a)\in H^{-\frac{n}{2}-k-1}(M)$ belongs to $C^k$ (which can be justified by using the Fourier transform), the map
\begin{align*}
(a,b)\in M^2\to (\d_{a},A\d_{b})
\end{align*}
also belongs to $C^k$. This implies that the Schwarz kernel of $A$ is smooth.

\end{proof}

By Corollary \ref{intsmcor} and Lemma \ref{intsmlem}, we conclude that $E_t$ is smooth under Assumption \ref{assresfree}. This completes the proof of Theorem \ref{mainth1}.

\section{From quasimode to breaking of smoothing effect}\label{sec3}

In this section, we prove Theorem \ref{mainth2}. To do this, we use a well-known fact that there is a quasimode which is concentrated near an elliptic geodesic.
For its proof, see \cite[\S 3.6]{C} or \cite[\S 13]{RT}. In \cite[\S 3.6]{C}, the Weyl law is used for the proof. On the other hand, the construction in \cite[\S 13]{RT} is based on the trapped geodesic in sphere.

\begin{lem}\label{quasiex}
There exist a Riemmanian manifold $(M,g)$ which is equal to Euclidean space $\re^n$ outside a compact set such that the following holds. There exist a sequence $h_k\in (0,1]$, $u_k\in L^2(M)$ and $c>0$ such that $h_k\to 0$ and
\begin{align*}
(-h_k^2\Delta_g-1)u_{k}=O_{L^2}(h_k^{\infty}),\quad \|u_k\|_{L^2(M)}=1,\quad \supp u_k\subset K
\end{align*}
with a compact set $K\subset M$ independent of $h_k$. 
\end{lem}

The next proposition shows that the smoothness of $E_t$ breaks when a quasimode with width less than two exists.

\begin{prop}\label{quasiimpnotsm}
Suppose that there exist a family $\{u_{h_k}\}_{k=1}^{\infty}\in L^2(M)$ with $h_k\to 0$ and a compact set $K\subset M$ independent of $h_k$ such that
\begin{align*}
(-h_k^2\Delta_g-1)u_{h_k}=o_{L^2}(h_k^{2}),\quad \|u_{h_k}\|_{L^2(M)}=1,\quad \supp u_{h_k}\subset K.
\end{align*}
Then, for each $t\in \re$, $E_t$ is not smooth.
\end{prop}

\begin{proof}
Since $E_t(x-y)=\d(x-y)$, which is not smooth, we may assume $t\neq 0$. If $E_t$ is smooth, then the operator $\chi e^{it\Delta_g}\chi$ maps from $L^2(M)$ to $H^m(M)$ continuously for each $m>0$ and $\chi\in C_c^{\infty}(M)$. Thus, it suffices to prove that there is $\chi\in C_c^{\infty}(M)$ such that $\chi e^{it\Delta_g}\chi\notin B(L^2(M), H^m(M))$ for each $m>0$.

Let $\chi\in C_c^{\infty}(M;[0,1])$ and $\f\in C_c^{\infty}((\frac{1}{2},2))$ such that $\chi(x)=1$ for $x\in K$ and $\f(E)=1$ near $E=1$. We denote $h=h_k$.
 First, we show 
\begin{align}\label{eneloc}
\f(-h^2\Delta_g)u_h=u_h+o_{L^2}(h^{2}).
\end{align}
Set $f_h=(-h^2\Delta_g-1)u_h$. Then $e^{-i\frac{s}{h}}u_h$ satisfies the Schr\"odinger equation
\begin{align*}
(ih\pa_s+h^2\Delta_g)(e^{-i\frac{s}{h}}u_h)=-e^{-i\frac{s}{h}}f_h.
\end{align*}
Then, the Duhamel formula implies
\begin{align}\label{Duh}
e^{-i\frac{s}{h}}u_h=e^{ish\Delta_g}u_h+\frac{i}{h}\int_0^s e^{-i\frac{r}{h}}e^{ih(s-r)\Delta_g}f_hdr.
\end{align}
Since
\begin{align}\label{Duhlow}
\|\frac{i}{h}\int_0^s e^{-i\frac{r}{h}}e^{ih(s-r)\Delta_g}f_hdr\|_{L^2(M)}\leq \frac{|s|}{h}\|f_h\|_{L^2(M)}=o(h|s|),
\end{align}
we have
\begin{align*}
\f(-h^2\Delta_g)u_h=&\frac{1}{2\pi h}\int_{\re}e^{\frac{is}{h}h^2\Delta_g}u_h\hat{\f}(\frac{s}{h})ds
=\frac{1}{2\pi h}\int_{\re}e^{-\frac{is}{h}}u_h\hat{\f}(\frac{s}{h})ds+\frac{o_{L^2}(h)}{2\pi h}\int_{\re}|s\hat{\f}(\frac{s}{h})|ds\\
=&\f(1)u_h+\frac{o_{L^2}(h^{2})}{2\pi }\int_{\re}|s\hat{\f}(s)|ds
=u_h+o_{L^2}(h^{2}),
\end{align*}
which shows $(\ref{eneloc})$.

Set $A_h=\chi(x)\f(-h^2\Delta_g)$. By $(\ref{Duh})$ and $(\ref{Duhlow})$, we have
\begin{align*}
\|A_he^{ish\Delta_g}u_h\|_{L^2(M)}\geq& \|e^{-i\frac{s}{h}}A_hu_h\|_{L^2(M)}-\frac{1}{h}\|A_h\int_0^s e^{-i\frac{r}{h}}e^{ih(s-r)\Delta_g}f_hdr\|_{L^2(M)}\\
\geq&\|A_hu_h\|_{L^2(M)}-\frac{|s|}{h}\|A_h\|_{L^2\to L^2} \|f_h\|_{L^2(M)}=\|u_h\|_{L^2(M)}+o(h^{2})+o(|s|h)\\
\geq &1+o(h^{2})+o(|s|h),
\end{align*}
where we use $\chi u_h=u_h$.
Taking $s=t/h$, we obtain
\begin{align*}
\|A_he^{it\Delta_g}u_h\|_{L^2(M)}\geq 1+o(1).
\end{align*}
By the relation $\chi u_h=u_h$, we obtain $\|A_he^{it\Delta_g}\chi u_h\|_{L^2(M)}\geq \frac{3}{4}$ if $h>0$ is sufficiently small. Moreover, since $[\chi,\f(-h^2\Delta_g)]=O_{L^2\to L^2}(h)$, we also obtain
\begin{align*}
\|\f(-h^2\Delta_g)\chi e^{it\Delta_g}\chi u_h \|_{L^2(M)}\geq \frac{1}{2}
\end{align*}
for $h>0$ sufficiently small. Thus, we conclude
\begin{align*}
\|\chi e^{it\Delta_g}\chi u_h \|_{H^m(M)}\geq& Ch^{-2m}\|\f(-h^2\Delta_g)\chi e^{it\Delta_g}\chi u_h \|_{L^2(M)}\\
\geq&2^{-1}Ch^{-2m}\to \infty
\end{align*}
as $h\to 0$, which implies $\|\chi e^{it\Delta_g}\chi\|_{L^2(M)\to H^m(M)}=\infty$.

\end{proof}

Now Theorem \ref{mainth2} follows from Lemma \ref{quasiex} and Proposition \ref{quasiimpnotsm}.

\begin{rem}\label{compnotsm}
If $(M,g)$ is a compact Riemmanian manifold, it follows that $E_t$ is not smooth. In fact, there exists a sequence $h_k\in(0,1]$ and $u_k\in L^2(M)$ such that $h_k\to 0$ and
\begin{align*}
(-h_k^2\Delta_g-1)u_k=0,\quad \|u_k\|_{L^2(M)}=1.
\end{align*}
Thus, if $m>0$, then $\|e^{it\Delta_g}u_k\|_{H^m(M)}= \|(1+h_k^{-2})^{\frac{m}{2}}u_k\|_{L^2(M)}=(1+h_k^{-2})^{\frac{m}{2}}\to \infty$ as $k\to \infty$. This implies $e^{it\Delta_g}\notin B(L^2(M), H^m(M))$. Since $M$ is compact, we conclude that $E_t$ is not smooth. 
\end{rem}

\end{document}